\documentclass[reqno]{amsart}
\usepackage{bbm}
\usepackage{amsfonts}
\usepackage{mathrsfs}
\usepackage{hyperref}
\usepackage{amssymb}
\usepackage{amsmath,amscd}
\usepackage{CJK}
 \newtheorem{thm}{Theorem}[section]
 \newtheorem{cor}[thm]{Corollary}
 \newtheorem{prop}[thm]{Proposition}
 
 \newtheorem{conj}[thm]{Conjecture}
 \newtheorem{lem}[thm]{Lemma}
 \newtheorem{ex}[thm]{Example}
 \newtheorem{defn}[thm]{Definition}
 \theoremstyle{remark}
 
 \numberwithin{equation}{section}

\def\bi{\binom}
\begin{document}
\title[Asymptotic estimate for the polynomial coefficients]
  { Asymptotic estimate for the polynomial coefficients}

\author{Jiyou Li}
\address{Department of Mathematics, Shanghai Jiao Tong University, Shanghai, P.R. China}
\email{lijiyou@sjtu.edu.cn}


\begin{abstract}

The polynomial coefficient $\bi {n,q}{k}$ is defined to be the
coefficient of $x^{k}$ in the expansion of $(1+x+x^2+\cdots
+x^{q-1})^n$. In this note we give an asymptotic estimate for $\bi
{n,q}{cn}$ as $n$ tends to infinity, where $c$ is a positive
integer.  Based on experimental results, it was conjectured that for
any $n$, $\bi {n,q}{cn}-\bi {n,q-1}{cn}$ is unimodal and its maximum
value occurs $q=\lfloor\log_{1+\frac 1{c}}{n}\rfloor$ or
$q=\lfloor\log_{1+\frac 1{c}}{n}\rfloor+1$. In particular, when
$c=1$, its maximum value occurs for $q=\lfloor\log_2{n}\rfloor$ or
$q=\lfloor\log_2{n}\rfloor+1$.

%


\end{abstract}

\maketitle \numberwithin{equation}{section}
\newtheorem{theorem}{Theorem}[section]
\newtheorem{lemma}[theorem]{Lemma}
\newtheorem{example}[theorem]{Example}
\allowdisplaybreaks

\section{Introduction}

 The polynomial coefficient, $\bi {n,q}k$, defined to
be the coefficient of $x^k$ in the expansion of $(1+x+x^2+\cdots
+x^{q-1})^n$. That is,
$$\sum_{k=0}^\infty \bi {n,q}k x^k=(1+x+x^2+\cdots +x^{q-1})^n.$$
In the case $q=2$  it is the binomial coefficient $\bi {n}k$ and in
the case $q=3$ it is the trinomial coefficient. Clearly $\bi {n,q}k$
can be regarded as a generalization of binomials and is one of the
fundamental combinatorial coefficients. It was studied extensively
by many mathematians since the time of Euler's. For details we refer
to \cite{An,C,E,OR}.  Some applications in coding theory and
communication theory can be found in \cite{Gi,GO,Liu,Luo}.

The polynomial coefficient has close relation with a compostion. Let
$k$ be a positive integer. A composition (also called ordered
partitions) of $k$ is a finite sequence of positive integers $x_1,
x_2,\cdots, x_r$ such that $x_1+x_2+\dots+x_r = k$. The $x_i$ are
called parts of the composition. Let $b(k,n,q)$ be the number of
compositions of $k$ with $n$ parts such that each part is bounded by
$q$. Thus $b(k,n,q)$ is exactly the coefficient of $x^k$ in the
expansion of  $(x+x^2+\cdots +x^q)^n=x^n(1+x+\cdots +x^{q-1} )^n$
and $b(k,n,q)=\bi {n,q}{k-n}$. It also equals the number of distinct
ways in which $k$ identical balls can be distributed in $n$ labeled
boxes with each box containing at most $q-1$ balls and being
nonempty. This classical bounded model of compositions also has many
other combinatorial meanings such as restricted multi-combinations
and forbidden $0,1$-sequences.

It follows by above discussion that thee combinatorial properties of
both $b(k,n,q)$ and $\bi {n,q}k$ are essentially the same. In this
paper we will focus on the study of $\bi {n,q}k$. In \cite{PWZ}, it
was proved that when $q>2$, ${n,q\choose k}$ has no closed form,
that is, it can not be expressed to be the sum of a fixed number of
hypergeometric terms.  Thus it is a natural question to ask some
asymptotical estimate. There seems to have many questions on the
distributions of $\bi {n,q}k$.

Denote $f(x)\sim g(x)$ if $$\lim_{x \rightarrow \infty} \frac
{f(x)}{g(x)}=1.$$

In the case $q=2$, it is well known that the binomial coefficient
$\bi {n,2}{cn}=\bi {n}{cn}$ has asymptotical estimate
$$\bi {n,2} {cn}\sim \frac 1 {\sqrt{2 \pi (c-c^2) n}}({c^{-c}(c-1)^{-c+1}})^n,$$
where $0<c<1$ is a constant.

In the case $q=3$, it is well known that the trinomial coefficient
has asymptotical estimate
$$\bi {n,3}n\sim \frac {3^{n+1/2}}{2 \sqrt{\pi n}}.$$

 The first result of this paper is an asymptotical estimate for ${n,q\choose
 k}$ for general $q>3$ by applying Hayman's method in the case $k=cn$ and $c$ is fixed positive integer.

\begin{thm}\label{Theorem 1.2}
Suppose $q>3$ and $k=cn$, where $c<q$ is an absolute positive integer.
Then we have
$$
 {n,q \choose cn }\sim \frac {\phi (r)}{\sqrt{2\pi n}}\left ( \frac {1-r^q} {r-r^2} \right )^n,$$
as $n\rightarrow\infty$, where
$$
\phi(r)=\left(\frac r{(1-r)^2}-\frac {q^2 r^{q}}
{(1-r^q)^2}\right)^{-1/2},
$$
$$
r=\frac {1}{d} +\frac {q}{c^2 d^{q+2}}+\theta \frac {q^3}{d^{2q}},
$$
 $|\theta|\leq 1$ and $d=1+\frac 1 c$.
In particular, when $c=1$ we have
 $$
r={\displaystyle \frac {1}{2}}  + {\displaystyle \frac {q}{2^{q + 2}
}}+ \theta \frac{q^3}{2^{2q}}.
$$
\end{thm}


 Star \cite{Star} gave an asymptotic formula only for the case
$k=\frac 12(n-s)(q+1)$, where $s=Kn^{\theta}, 0\leq \theta\leq 1/2$ and
$K$ is a positive constant.

We are interested in the unimodality of the polynomial coefficients.
A sequence $\{ a_0,a_1,\cdots,a_n\}$ is \textbf{unimodal} if there
exits index $k$ with $0 \leq k\leq n$ such that
$$a_0\leq a_1\leq \cdots a_{k-1}\leq a_k \geq a_{k+1} \cdots\geq a_n.$$
A sequence $\{ a_0,a_1,\cdots,a_n\}$ is called reciprocal if $a_i
=a_{n-i}$ for $0\leq i<n$.

\begin{prop}
 For given $n, k$, $n,q \choose k$ is unimodal and reaches its maximum at $k=qn/2$ provided $qn$ is
even and $k=(qn-1)/2$ provided $qn$ is odd.
\end{prop}

This well known proposition has many proofs. Perhaps the most simple
one is that the product of two unimodal reciprocal polynomials is
also unimodal reciprocal, for details please refer to \cite{An}.

Let $a(k,n,q)$ be the number of compositions of $k$ with $n$ parts
such that the largest part is $q$. Clearly
$a(k,n,q)=b(k,n,q)-b(k,n,q-1)$ and thus $a(2n,n,q)=\bi {n,q}n-\bi
{n,q-1}n$.

Let $b(k,q)=\sum_{n=1}^k b(k,n,q)$ be the number of all compositions
of a positive integer $k$ with parts bounded by $q$ and
$a(k,q)=\sum_{n=1}^k a(k,n,q)$ be the number of all compositions of
a positive integer $k$ such that the largest part is $q$. It is well
known that
$$\sum_{k=0}^{\infty}b(k,q)x^k=\frac {1-x}{1-2x+x^{q+1}}.$$

  Based on this formula and analytical tools, Odlyzko and Richmond \cite{OR} proved the following theorem.
\begin{thm}[Odlyzko and Richmond]
Let $a(k,q)$ be the number of all compositions of a positive integer
$k$ with parts bounded by $q$. Then  $a(k,q)$ is unimodal for any
$k$ and the maximum value occurs for $q=\lfloor \log_2{k}\rfloor$
infinitely often and $q=\lfloor \log_2{k}\rfloor+1$ infinitely often
and always at one of these two values and no other.
\end{thm}

Our conjecture is a more subtle one:
\begin{conj}\label{Theorem 1.1}
 Let $a(k,n,q)$ be the number of compositions of $k$ with $n$ parts
such that the largest part is $q$. Let $c$ be a positive integer.
Then for any $n$, $a((c+1)n,n,q)=\bi {n,q}{cn}-\bi {n,q-1}{cn}$ is a
unimodular function on $q$ and the maximum value occurs for
$q=\lfloor\log_{1+\frac 1c}{n}\rfloor+1$ or $q=\lfloor\log_{1+\frac
1c}{n}\rfloor+1$.

In particular, $a(2n,n,q)=\bi {n,q}n-\bi {n,q-1}n$ is a unimodular
function on $q$ and the maximum value occurs for
$q=\lfloor\log_2{n}\rfloor$ or $q=\lfloor\log_2{n}\rfloor+1$.
\end{conj}

The paper is organized as follows. A brief (but not complete)
historical review is given in Section 2 and the main results are
proved in Section 3. For simplicity of the computations, we only
gives the proof for the special case $k=n$ i.e., $c=1$. In the last
Section, we give some analysis to support Conjecture \ref{Theorem
1.1}.

%


\section{Historical results on ${n,q \choose k }$}

Many equalities on the polynomial coefficients were found since the
time of Euler's \cite{E}. For instance, the following equations are
well-known and proofs can be found in \cite{Li}:

\begin{prop}
  $ { n,q \choose k }$ satisfies:
   \begin{align*}
 { n,q \choose k }&=\sum_{i=0}^{q-1}{ n-1,q \choose k-i };\\
{ n,q \choose k }&={ n-1,q \choose k }+{ n,q \choose k-1 }-{
  n-1,q \choose k-q };\\
  { n,q \choose k }&=\sum_{i=0}^n (-1)^i { n \choose i }{ n+k-iq-1
\choose {n-1} }; \\
  { n,q \choose k }&=\frac 2 \pi \int_0^{\frac \pi 2} (\frac {\sin
q\theta } {\sin\theta } )^n \cos((q-1)n -2k))\theta) d \theta . \ \ \
\end{align*}
\end{prop}
%

Andr\'{e} proved that \cite{C}
$$
\sup_k{{ n,q \choose k }}\sim q^n\sqrt{\frac 6 {(q^2-1)\pi n}}, \ \
n\rightarrow \infty.
$$

Star \cite{Star} gave an asymptotic formula  only for the case
$k=\frac 12(n-s)(q+1)$, where $s=Kn^{\theta}, 0\leq \theta\leq 1/2$ and $K$
is a positive constant. Note that this formula generalized the
result of Andr\'{e}.
\begin{thm}[Star]
Let $k=\frac 12(n-s)(q+1)$, where $s=Kn^{\theta}, 0\leq \theta<1/2$ and $K$
is a positive constant.
As $n$ tends to infinity,
 \begin{align*}{ n,q+1 \choose k-n} &=\frac 1 {\sqrt{\pi}}\left(
 \frac 6 {q^2-1}\right)^{1/2}\frac {q^n}{n^{1/2}}\cdot \\
&\left(
 1+\frac {h_{1,0}(q)+h_{1,1}(q)s^2}{n}+\cdots+
 \frac {\sum_{j=0}^{m-1} h_{m-1,j}(q)s^{2j}}{n^{m-1}}+O(\frac
{1+s^{2m}}{n^m})\right),
 \end{align*} where $h_{i,j}$ are some rational
functions of $q$.
\end{thm}

\section{Main result}

\begin{defn}

Suppose that $f(z)=\sum_{n=0}^{\infty} a_n z^n$ is a complex
analytic function for $|z|<R$, where $0<R\leq \infty$. Define
 \begin{align} \label{5}
M(r)=\max_{|z|=r}{|f(z)|}.
 \end{align}
If for large enough $r$, we have $M(r)=f(r)$, then $f(z)$ is called
an admissible function. Please refer to \cite{H,W} for detailed
theory on admissible functions.
\end{defn}

Note that this one is different from the  current definition of
admissible functions, which actually defines a function satisfying
(3.1).

In the well known paper \cite{H} Hayman proved that such good
functions have very good asymptotical estimate on their
coefficients. The following lemma due to Hayman gives a subtle
estimate on controlling its coefficients for an admissible
analytical function.
\begin{lem}[Hayman] \label{lem2}

 Let $f(z)=\sum_{n=0}^{\infty} a_n z^n$ be an admissible function, which is analytic in the disk
 $|z|<R$. Denote
 $$
a(r)=r\frac {f'(r)}{f(r)},\ \ \ b(r)=ra'(r),
$$
and suppose $0<r_n<R$ is a positive real root satisfying
$$ a(r_n)=n,\
\ \ \forall n \in N.
$$
Then
$$
a_n \sim \frac {f(r_n)}{r_n^n\sqrt{2\pi b(r_n)}},\ \ \ n \rightarrow
\infty .
$$
\end{lem}

\begin{ex}
  $f(z)=(1+z+z^2+\cdots+z^{q-1})^n$ is an
admissible function analytical in the disk $|z|<1$.
\end{ex}

\begin{lem}\label{lem1}
Assume $q\geq 3$. Then the equation
$$
(q-2)x^{q+1}-(q-1)x^{q}+2x-1=0, \ \  \ q \in N $$ has only two
positive real roots including 1 as a trivial one. The nontrivial one
$r$ satisfies
$$\big| r-\frac {1}{2}-\frac {q}{2^{q + 2}
}\big |\leq \frac{q^3}{2^{2q}}.
$$

\end{lem}

\begin{proof}
Since the cases $q=3, 4$ can be verified directly, we may assume
$q>4$.  Suppose $f(x)=(q-2)x^{q+1}-(q-1)x^{q}+2x-1$. Then $f''(x)=q
x^{q-2}(xq^2-2x-xq-q^2+2q-1)=0$ gives two inflection points $\frac
{(q-1)^2}{(q+1)(q-2)}, 0$. This proves that there are only two
positive real roots including 1.

 Suppose now $r=\frac 1
2 +\frac {c}{2^{q+2}}$ is a positive real root of $f(x)$, where $c$
is regarded as a variable depending on $q$ and will be specified.

By the definition we have
 \begin{align*} f(r)
&=(q-2)(\frac 1 2 +\frac {c}{2^{q+2}} )^{q+1}-(q-1)(\frac 1 2 +\frac {c}{2^{q+2}} )^{q}+2(\frac 1 2 +\frac {c}{2^{q+2}} )-1\\
&=(\frac {-q} {2^{q+1}}+ \frac {(q-2)c} {2^{2q+2}}  )(1 +\frac
{c}{2^{q+1}})^{q}+\frac {c}{2^{q+1}}\\
&=0.
 \end{align*}
   We assume without of generality that $0\leq c\leq q^{3/2}$. By the Taylor's expansion, for any $q$ we
have,
 \begin{align*}
\big|(1 +\frac {c}{2^{q+1}})^{q}-1-{\frac {cq}{2^{q+1}}\big|\leq
\frac {c^2 q^2}{2^{2q}} \leq \frac {q^5}{2^{2q}} }.
 \end{align*}
And hence
 \begin{align*}
\big | f(r)+\frac {q-c} {2^{q+1}}+\frac {c(q^2-q+2) } {2^{2q+2}}\big
|\leq\frac {q^6}{2^{3q+1}}.
 \end{align*}

Since $f(r)=0$, we then get
 $$
 q-c+\frac {c(q^2-q+2) } {2^{q+1}}+ \frac {\theta q^6}{2^{2q}}=0,
 $$
 where $0\leq |\theta|\leq 1$.
Set  $c=q+c'$ and substitute this into the above equality one has
$$
c'(1-\frac {q^2-q+2}{2^{q+1}})=\frac {q^3-q^2+2q}{2^{q+1}}+\frac
{\theta q^6}{2^{2q}}.
$$
Finally we get that
$$
\big | c-q-\frac {q^3-q^2+2q}{2^{q+1}-q^2+q-2}\big |\leq \frac{2
 q^6}{2^{2q}}.
 $$
 Substitute this into $r=\frac 1 2
+\frac {c}{2^{q+2}}$ and when $q>16$ we then have
$$\frac{q^3-q^2}{2^{2q+2}} \leq  r-\frac {1}{2}-\frac {q}{2^{q + 2}
}\leq \frac{q^3}{2^{2q+2}}.
$$ The cases for $3<q<16$ can be easily checked by computers
and thus the proof is complete.
\end{proof}

\begin{thm}\label{Theorem 3.4}
Assume $q>3$. Then we have the asymptotic estimate $$
 {n,q \choose n }\sim \frac {\phi (r)}{\sqrt{2\pi n}}\left ( \frac {1-r^q} {r-r^2} \right )^n, \ \ n\rightarrow \infty$$
where
$$
\phi(r)=\left(\frac r{(1-r)^2}-\frac {q^2 r^{q}}
{(1-r^q)^2}\right)^{-1/2},
$$
$$\big| r-\frac {1}{2}-\frac {q}{2^{q + 2}
}\big |\leq \frac{q^3}{2^{2q}}.
 $$
\end{thm}

\begin{proof}
Let $f(z)=(1+z+z^2+\cdots+z^{q-1})^n=(\frac {1-z^q}{1-z})^n.
  $ One checks that $f(z)$ is indeed an admissible analytical function on $\mathbb{C}-\{\infty\}$.
  Applying Hayman's theorem Lemma \ref{lem2} we have
 \begin{align*}
  a(x)&=x\frac {f'(x)}{f(x)}=\frac {-nx(qx^{q-1}-qx^q-1+x^q) } {(1-x^q)(1- x)},\\
b(x)& =xa'(x)=\frac {nx(1-x^{q-1}q^2-2x^q+x^{2q}+2q^2x^q-x^{q+1}q^2) }{(-1+x^q)^2(x-1)^2}\\
 &=nx\left(\frac 1 {(1-x)^2}-\frac {q^2 x^{q-1}} {(x^q-1)^2}\right).
 \end{align*}

Consider the equation $a(x_n)=n$, we then have
 $$
 \frac {-nx_n(qx_n^{q-1}-qx_n^q-1+x_n^q) } {(1-x_n^q)(1- x_n)}=n,
 $$
and thus  $$ (q-2)x_n^{q+1}-(q-1)x_n^{q}+2x_n-1=0. $$ Applying Lemma
\ref{lem1} we obtain the desired formula.
\end{proof}

\begin{cor}
When $q>3$, for large $n$ we have the estimate
$$
 {n,q \choose n }\sim \frac
  {(1+\frac{q^2-6q}{2^q}+\theta_1\frac{q^2}{2^{2q}}) 2^n}{\sqrt{\pi n}}(1-\frac
1{2^{q-2}}+\theta_2\frac{q^2}{2^{2q}})^n, \  n\rightarrow \infty,$$
where $|\theta_i|\leq1$ for $i=1,2$.
\end{cor}
\begin{proof}
$$\big| r-\frac {1}{2}-\frac {q}{2^{q + 2}
}\big |\leq \frac{q^3}{2^{2q}}.
$$
and thus one computes that
$$\big |\frac r {(1-r)^2}-\frac 12- \frac {3q}{2^{q-1}}\big|\leq \frac{q}{2^{2q}},$$
and
$$\big |\frac {q^2r^q}
{(1-r^q)^2}-\frac{q^2}{2^q}\big |\leq \frac{q^4}{2^{2q}}.$$ Thus
$$
\big |\phi(r)-\sqrt{2}(1+\frac{q^2-6q}{2^q})\big|\leq
\frac{q^2}{2^{2q}}.
$$
Similarly we have $$\big|\frac {1-r^q} {r-r^2}-2+\frac
1{2^{q-1}}\big|\leq \frac{q^2}{2^{2q}},$$ and the formula follows
from Theorem \ref{Theorem 3.4}.
\end{proof}

Fix another positive integer $c$. The same method for $k=cn$ gives a
general result. We omit the details since the proof of this
generalization is essentially the same as the case $c=1$.
 For more details, please refer to \cite{Li}.
\begin{thm}
$$
 {n,q \choose cn }\sim \frac {\phi (r)}{\sqrt{2\pi n}}\left ( \frac
  {1-r^q} {r-r^2} \right )^n, \ \ n\rightarrow \infty,$$
where
$$
\phi(r)=\left(\frac r{(1-r)^2}-\frac {q^2 r^{q}}
{(1-r^q)^2}\right)^{-1/2},
$$ and
$$
r=\frac {1}{d} +\frac {q}{d^{q+2}}+O(\frac {q^2}{d^{2q}}),
$$
 and $d=1+\frac 1 c$.
\end{thm}
\section{Unimodality}

A sequence $\{ a_0,a_1,\cdots,a_n\}$ is \textbf{unimodal} if there
exits index $k$ with $0 \leq k\leq n$ such that
$$a_0\leq a_1\leq \cdots a_{k-1}\leq a_k \geq a_{k+1} \cdots\geq a_n.$$

We then consider the unimodality on the difference sequence of $\bi
{n,q}n$ on $q$.
\begin{conj} The sequence $\bi {n,q}n-\bi {n,q-1}n$ is a unimodular
function on $q$ and the maximum value occurs for
$q=\lfloor\log_2{n}\rfloor$ or $q=\lfloor\log_2{n}\rfloor+1$.

Furthermore, for a positive integer $c$, the sequence $\bi
{n,q}{cn}-\bi {n,q-1}{cn}$ is a unimodular function on $q$ and the
maximum value occurs for $q=\lfloor\log_{1+\frac 1c}{n}\rfloor+1$ or
$q=\lfloor\log_{1+\frac 1c}{n}\rfloor+1$.

\end{conj}

We can not prove this conjecture. Instead, we prove that the
difference sequence of the main part of the asymptotic estimate on
$\bi{n,q}n$ is unimodal on $q$ when $n$ is taken to be large enough
and fixed. This may be regarded as an evidence why the above
conjecture should hold.

For a large integer $n$, denote $$f(q)=\frac
  {(1+\frac{q^2-6q}{2^q}+\theta_1\frac{q^2}{2^{2q}}) 2^n}{\sqrt{\pi n}}(1-\frac
1{2^{q-2}}+\theta_2\frac{q^2}{2^{2q}})^n, $$ where $|\theta_i|\leq1$
for $i=1,2$. Let $g(q)=f(q+1)-f(q)$.

 \begin{cor}
For large enough $n$, the sequences $g(q)$ is unimodular with
maximum at $q=\lfloor \log_2{n}\rfloor$ or $q=\lfloor
\log_2{n}\rfloor+1$.
\end{cor}

 \begin{proof}

 It is clear that the unimodality of $g(q)$ is equivalent to
 consider the unimodality of  $T(q)=B(q+1)-B(q)$, where
$$B(q)=(1+\frac{q^2-6q}{2^q}+\theta_1\frac{q^2}{2^{2q}})(1-\frac
1{2^{q-2}}+\theta_2\frac{q^2}{2^{2q}})^n.$$ One computes
 \begin{align*}
T(q+1)-T(q)&=B(q+2)+B(q)-2B(q+1)\\
&=(1+\frac{q^2-2q-8}{2^{q}}+\theta_1\frac{q^2}{2^{2q}})(1-\frac
1{2^{q+2}}+\theta_2\frac{q^2}{2^{2q}})^n\\
&+ (1+\frac{q^2-6q}{2^{q-2}}+\theta_1\frac{q^2}{2^{2q}})(1-\frac
1{2^q}+\theta_2\frac{q^2}{2^{2q}})^n\\
&-2(1+\frac{q^2-4q-5}{2^{q+1}}+\theta_1\frac{q^2}{2^{2q}}) (1-\frac
1{2^{q-1}}+\theta_2\frac{q^2}{2^{2q}})^n.
 \end{align*}

 It then suffices to prove that if
$q\leq  \log_2{n}$, then $T(q+1)-T(q)<0$ and if $q\geq \log_2{n}+1$,
then $T(q+1)-T(q)>0$.
  Assume $n$ is large, since $|\theta_i|\leq 1$, if $q$ is small enough, say $q\ll (\log_2{n})^{\epsilon}$
  for some small positive constant $\epsilon$,
  we then have $T(q+1)-T(q)>0$. The interesting cases happen when $q$
  is very close to $\log_2{n}$.
  Not that when $q=\log_2{n}$, by the formula
  $$(1-\frac 1n)^n<e^{-1}<(1-\frac 1{n+1})^n$$
  we then have $B(\log_2{n}+i)\sim e^{-1/2^i}$ and thus when $n$ is large
  enough,
$$T(\log_2{n}+1)-T(\log_2{n})\sim e^{-1/4}+e^{-1}-2e^{-1/2}<0.$$
Since the function
$$ \frac {e^{-2^{i+1}}+e^{-2^{i-1}}}{2 e^{-2^i}}=\frac 12 e^{-2^i}+\frac 12 e^{-2^{i-1}}<1$$
for all $i\leq -1$, which means
$e^{-2^{i+1}}+e^{-2^{i-1}}-2e^{-2^i}<0$ for all $i\leq -1$. This
shows that when $q\geq \log_2{n}+1$, $T(q+1)-T(q)<0$.

 Similarly, since the function
$$ \frac {e^{-2^{i+1}}+e^{-2^{i-1}}}{2 e^{-2^i}}=\frac 12 e^{-2^i}+\frac 12 e^{-2^{i-1}}>1$$
for all $i\geq 0$, which means
$e^{-2^{i+1}}+e^{-2^{i-1}}-2e^{-2^i}<0$ for all $i\geq 0$. For
instance, $$T(\log_2{n})-T(\log_2{n}-1)\sim
e^{-1/2}+e^{-2}-2e^{-1}>0.$$ This shows that if $q\leq \log_2{n}$,
then $T(q+1)-T(q)>0$.

\end{proof}

\begin{cor}
The sequences $b(cn,n,q)$ is unimodular and reaches at its maximum at $q=\lfloor
\log_{1+\frac 1 c}{n}\rfloor$ or $q=\lfloor \log_{1+\frac 1
c}{n}\rfloor+1$.
\end{cor}
We are expecting a combinatorial proof of our result. Recall $b(k,n,q)$ is
the number of compositions of $k$ with $n$ parts such that the
largest part equals $q$. Let $c>1$ be a positive absolute constant.
\begin{conj}
 Is there a combinatorial proof of the unimodality of $a(cn,n,q)$ on $q$?
 Equivalently, is there a combinatorial proof of the unimodality of
 $ {n,q \choose cn }-{n,q-1 \choose cn}$, where  ${n,q \choose cn
 }$ is the coefficient of $x^{cn}$ in the polynomial $(1+x+x^2+\cdots
 +x^{q-1})^n$?
\end{conj}

{\bf Acknowledgements.} Part of this work was done during the
author's study at Beijing Normal University.  The author wishes to
express his memory to Professor Wilf. In 2006 he asked him this
conjecture without knowing the work of Odlyzko and Richmond at that
time and had a nice talk at Nankai University. The author also
wishes to thank the referee for his many constructive comments.


%
%
%
%
%
%
%
%
%
%
%
\end{document}